\pdfoutput=1
\documentclass{amsart}
\usepackage{booktabs}
\usepackage{amssymb}
\makeindex

\usepackage{hyperref}
\hypersetup{
    colorlinks,%
    citecolor=blue,%
    filecolor=blue,%
    linkcolor=blue,%
    urlcolor=blue
}

 \usepackage[all]{xy}
\usepackage{tikz}
\usepackage{todonotes}
\usepackage[backrefs,lite]{amsrefs}
\usepackage{enumerate}
\usepackage{verbatim}
\usepackage{multirow}
\usepackage{pgf,tikz}
\usetikzlibrary{arrows}

\CompileMatrices
\newtheorem{introthm}{Theorem}

\newtheorem{theorem}{Theorem}[section]
\newtheorem{lemma}[theorem]{Lemma}

\newtheorem{proposition}[theorem]{Proposition}
\newtheorem{corollary}[theorem]{Corollary}
\theoremstyle{definition}
\newtheorem{definition}[theorem]{Definition}
\newtheorem{example}[theorem]{Example}

\newtheorem{construction}[theorem]{Construction}

\newtheorem{remark}[theorem]{Remark}

\def\cc{{\mathbb C}}

\def\zz{{\mathbb Z}}
\def\rr{{\mathbb R}}
\def\nn{{\mathbb N}}
\def\qq{{\mathbb Q}}
\def\pp{{\mathbb P}}

\def\Osh{{\mathcal O}}
\def\R{{\mathcal R}}
\def\Sl{\operatorname{SL}}
\def\Hom{\operatorname{Hom}} 
\def\Spec{\operatorname{Spec}}

\def\Conv{\operatorname{Conv}} 
\def\Cl{\operatorname{Cl}} 
\def\Aut{\operatorname{Aut}} 

\def\F{{\mathcal F}}

\begin{document}
\title[Calabi-Yau 
hypersurfaces in $\qq$-Fano toric varieties]{Families of Calabi-Yau 
hypersurfaces\\ in $\qq$-Fano toric varieties}

\author{Michela Artebani}
\address{
Departamento de Matem\'atica, \newline
Universidad de Concepci\'on, \newline
Casilla 160-C,
Concepci\'on, Chile}
\email{martebani@udec.cl}

\author{Paola Comparin}
\address{
Departamento de Matem\'atica, \newline
Universidad de Concepci\'on, \newline
Casilla 160-C,
Concepci\'on, Chile}
\email{pcomparin@udec.cl} 

\author{Robin Guilbot}
\address{
Instituto Nacional de Matem\'atica Pura e Aplicada, \newline
Estrada Dona Castorina, 110 \newline
Rio de Janeiro, Brasil}
\email{rguilbot@math.cnrs.fr}


\thanks{The first author has been partially 
supported by Proyecto Fondecyt Regular N. 1130572 and Proyecto Anillo ACT 1415 PIA
Conicyt.
The second author has been partially 
supported by Proyecto Fondecyt Postdoctorado N. 3150015 and Proyecto Anillo ACT
1415 PIA Conicyt. 
The third author was supported by the CNPq, 
Conselho Nacional de Desenvolvimento Cient\'{i}fico 
e Tecnol\'{o}gico, Brasil,  and by the NCN project 2013/08/A/ST1/00804.}

\begin{abstract} 
We provide a sufficient condition for a general 
hypersurface in a $\qq$-Fano toric variety  
to be a Calabi-Yau variety in terms of its Newton polytope.
Moreover, we define a generalization of the 
Berglund-H\"ubsch-Krawitz construction in case 
the ambient is a $\qq$-Fano toric variety with 
torsion free class group
and the defining polynomial is not necessarily of Delsarte type.
Finally, we introduce a duality between 
families of Calabi-Yau hypersurfaces 
which includes both Batyrev
and Berglund-H\"ubsch-Krawitz mirror constructions. 
This is given in terms of a polar duality 
between pairs of polytopes $\Delta_1\subseteq \Delta_2$,
where $\Delta_1$ and $\Delta_2^*$ 
are canonical.
 \end{abstract}

\maketitle

\section*{Introduction}

A wide class of Calabi-Yau varieties is given by anticanonical 
hypersurfaces, or more generally complete intersections, in Fano toric varieties. 
A special interest for such families of Calabi-Yau's arised 
after the work of Batyrev~\cite{Ba}, who defined a duality between 
the anticanonical linear series of toric Fano varieties 
which satisfies the requirement of topological mirror symmetry~\cites{Ba, BaBo}:
\begin{equation}\label{hodge}
h^{p,q}_{st}(X)=h^{n-p,q}_{st}(X^*),\quad 0\leq p,q\leq n,
\end{equation}
where $X,X^*$ are general elements in the dual linear series, 
$n=\dim(X)=\dim(X^*)$ and $h^{p,q}_{st}$ denote the string-theoretic Hodge numbers.
In the case of hypersurfaces, Batyrev mirror construction relies on the polar duality 
between reflexive polytopes.
A different class of Calabi-Yau varieties can be constructed by considering 
quotients of quasismooth (or transverse) hypersurfaces in weighted projective spaces, 
i.e. defined by homogeneous polynomials whose affine cone is singular only at the vertex.
For such Calabi-Yau varieties, in case they are defined by Delsarte type equations,
the physicists Berglund and H\"ubsch~\cite{BH} defined 
a transposition rule for the defining polynomial. 
The construction has been later refined by Krawitz \cite{K}, 
who introduced the action of finite diagonal symplectic groups. 
More precisely, the Berglund-H\"ubsch-Krawitz (BHK for short)
transposition rule is a correspondence
\[
\{W=0\}\subset \pp(w)/\tilde G \longleftrightarrow \{W^*=0\}\subset \pp(w^*)/\tilde G^*,
\]
where $W^*,w^*$ and $\tilde G^*$ are suitably defined transposed versions of the 
polynomial $W$, the set of weights $w$ and the group $\tilde G$  
(see Section \ref{bhk}).
Recently Chiodo and Ruan~\cite{CR} proved that the BHK 
transposition rule actually satisfies \eqref{hodge} in terms of the Chen-Ruan orbifold cohomology.

The motivation of this work is first to relate the two constructions 
of Calabi-Yau varieties described above and secondly 
to define a duality generalizing both Batyrev and BHK duality. 
A way of unifying the BHK duality with Batyrev-Borisov duality 
of reflexive Gorenstein cones has been proposed in \cite{Bor}, 
where the author develops the vertex algebras approach to mirror 
symmetry for the BHK construction and suggest 
a common framework for the latter and  Batyrev-Borisov duality (see \cite[\S 7]{Bor}). 
The present paper makes the previous framework more explicit 
in the hypersurface case in terms of polytopes 
and identifies the needed regularity condition 
to obtain a duality between families of Calabi-Yau varieties.

More precisely, given a $\qq$-Fano toric variety $X$ 
with canonical singularities, 
we consider families of anticanonical hypersurfaces of $X$ 
with fixed Newton polytope.
Called $\Theta$ the anticanonical polytope 
of $X$ and $\Delta$ the given Newton polytope,
one such family will be denoted by  $\mathcal F_{\Delta,\Theta^*}$.
We prove the following result,
where we recall that a canonical polytope is a lattice polytope whose 
unique lattice interior point is the origin.

\begin{introthm}\label{cy}
Let $X$ be a $\qq$-Fano toric variety with canonical singularities 
and let $\Delta\subset M_\qq$ be a  canonical polytope contained in the 
anticanonical polytope $\Theta$ of $X$.
Then a general hypersurface in $\mathcal F_{\Delta,\Theta^*}$ 
is a  Calabi-Yau variety.
\end{introthm}

This result gives examples 
of families of Calabi-Yau varieties in dimension $\geq 5$ 
whose general element is not quasismooth 
and it is not birational to a hypersurface in a toric Fano variety
 (see Table \ref{tableqs} for some of them).

Theorem \ref{cy} also suggests the definition 
of a duality between families of Calabi-Yau varieties 
with fixed Newton polytope in $\qq$-Fano toric varieties.
We will say that a pair $(\Delta_1,\Delta_2)$ of 
polytopes is a {\em good pair} if $\Delta_1\subseteq \Delta_2$ 
and both $\Delta_1$ and $\Delta_2^*$ are canonical.
Clearly the polar $(\Delta_2^*,\Delta_1^*)$ of a good pair is 
still a good pair.
This involution on good pairs produces the duality 
\[
\mathcal F_{\Delta_1,\Delta_2^*}\subseteq |-K_{X_{\Delta_2}}| 
\longleftrightarrow \mathcal F_{\Delta_2^*,\Delta_1}\subseteq |-K_{X_{\Delta_1^*}}|. 
\]
This coincides with Batyrev duality when $\Delta_1=\Delta_2$, 
since canonical polytopes whose polar is canonical are exactly reflexive polytopes.

Moreover, Theorem \ref{cy} allows to define 
a {\em generalized  BHK transposition rule} 
where the weighted projective space $\pp(w)$  
is replaced by a $\qq$-Fano toric variety 
with torsion free class group and canonical singularities. 
More precisely, a generalized BHK family $\mathcal F(A,\tilde G)$ 
in $X/\tilde G$ is a family of hypersurfaces of anticanonical degree
defined by a matrix $A$ of exponents for the equations and a 
symplectic group $\tilde G$.
We can associate to it a pair of polytopes $\Delta_1\subseteq \Delta_2$, where 
 $\Delta_1$ is the Newton polytope of a 
general element in $\mathcal F(A,\tilde G)$ and 
$\Delta_2$ is the anticanonical polytope of $X/\tilde G$. 
The following shows that the generalized BHK 
transposition rule can be seen as a duality between good pairs.

\begin{introthm}\label{thmbhk}
Let $(\Delta_1,\Delta_2)$ be the pair associated to a generalized Berglund-H\"ubsch-Krawitz family 
$\mathcal F(A,\tilde G)$. Then   $(\Delta_1,\Delta_2)$ is a good pair and 
the pair associated to $\mathcal F(A^T,G^*)$ is $(\Delta_2^*,\Delta_1^*)$.
\end{introthm}
The theorem relies on a toric description of 
the generalized BHK construction, which had been given 
in the classical case in \cite[\S 2]{Sh} and generalizes 
\cite[Proposition 2.9]{Sh}.

The paper is organized as follows. 
In Section \ref{background} we recall some definitions and basic results 
about toric varieties and polytopes. 
In Section \ref{sec-hyp} we study hypersurfaces in toric varieties and we describe 
their regularity properties according to the Newton polytope.
 In Section \ref{sec-duality} we prove Theorem \ref{cy} and we define the duality between good pairs. 
Section \ref{bhk} is devoted to the definition of the generalized BHK 
mirror construction and to the proof of Theorem \ref{thmbhk}. 
Finally Section \ref{hn} contains some remarks about the stringy Hodge numbers 
of our families.

\section{Toric background}\label{background}
\subsection{Toric varieties}
We start recalling some standard facts in toric geometry, see for example \cite{CLS}.    
Let $N$ denote a lattice and let $M={\rm Hom}(N,\zz)$ be its dual. 
Let $\Delta$ be a polytope in $M_{\qq}$, i.e. the convex hull of a finite subset of $M_\qq$. 
The {\em polar} of $\Delta$ is the polyhedron
\[
\Delta^*=\{y\in N_\qq: \langle x,y\rangle\geq -1,\ \forall\, x\in \Delta\},
\]
which clearly contains the origin 
in its interior and whose facets are contained 
in the affine hyperplanes of equation $\langle y,v_i\rangle=-1$, 
where $v_i$ is a vertex of $\Delta$.

It is well known that to any polytope $\Delta$ 
as above one can associate a toric variety 
$X=X_{\Delta}$ together with a $\qq$-divisor $D$.
The variety $X$ is the toric variety 
associated to the normal fan $\Sigma_{\Delta}$ to $\Delta$.
If $n_1,\dots,n_r$ are the primitive generators 
of the one-dimensional cones of $\Sigma_{\Delta}$ and $D_1,\dots,D_r$ 
are the corresponding integral torus-invariant divisors, 
then $D=-\sum_i h_{\Delta}(n_i)D_{i}$, where $h_{\Delta}$  
is the strictly upper convex function 
\[
h_{\Delta}:N_\qq\to \qq,\quad h_{\Delta}(y)=\min_{x\in \Delta}\{\langle x,y\rangle\}.
\]

Now let $P:\zz^r\to N$ be the homomorphism defined by $P(e_i)=n_i$,
which will be called $P$-{\em morphism} of the toric variety.
We denote by $P^T$ its transpose and by $Q$  the homomorphism 
defined by the following exact sequence:
\[
\xymatrix{
0\ar[r] & M\ar[r]^{P^T} &\zz^r\ar[r]^Q & K\ar[r] & 0,
}
\] 
where $K$ is isomorphic to the Class group of $X$.

The {\em Cox ring} $\mathcal R(X)$ is the polynomial ring $\cc[T_1,\dots,T_r]$, where 
$T_i$ is the defining element of the divisor $D_i$, graded by $K$:
$\deg(T_i)=Q(e_i)$. 
Let $\bar X=\Spec \mathcal R(X)\cong \cc^r$. 
 By Cox's construction~\cite[Theorem 5.1.11]{CLS} a toric variety $X$ 
associated to a fan $\Sigma$ can be described as a GIT-quotient  of a quasitorus
\[
X\cong \left( \bar X \setminus V(I) \right) \mathbin{/\! /} G_\Sigma,
\]
where $V(I)=V(x^{\hat{\sigma}} \mid \sigma \in \Sigma)$  
is the irrelevant locus, that is the subvariety of $\bar X$ 
defined by the vanishing of every monomial of the form 
$x^{\hat{\sigma}}= \prod_{\rho \notin \sigma(1)} x_\rho$ for $\sigma \in \Sigma$, and 
\[
G_\Sigma=\Hom(\Cl(X),\cc^*) \simeq \Spec(\cc[\Cl(X)]).
\] 
The big torus $\hat{T}=(\cc^*)^r$ naturally acts on the characteristic space 
$\hat{X}=\bar X\setminus V(I)$  by coordinatewise multiplication. 
Observe that $G_{\Sigma}$ is the kernel of the natural homomorphism
\[
\Psi:\hat T\to T,\quad (\lambda_1,\dots, \lambda_r)\mapsto (u\mapsto \prod_{j=1}^r\lambda_j^{\langle u,n_j\rangle}),
\]
where we identify $T$ with ${\rm Hom}(M,\cc^*)$.

We finally recall how quotients of toric varieties by finite subgroups of the torus can be described.
Let $X=X_{\Sigma,N}$ be a toric variety associated to a fan $\Sigma\subset N_{\qq}$
and let $N\to N'$ be a lattice monomorphism with finite cokernel 
$G$. The fan $\Sigma$ in $N'_{\qq}=N_\qq$ defines a toric variety $X'$ and
the inclusion of lattices induces a morphism of toric varieties 
$X\to X'$ which is a good geometric quotient by the action of the group $G$ 
\cite[Proposition 3.3.7]{CLS}.

\begin{lemma}\label{quot} Let $X=X_{\Sigma,N}$ be a toric variety associated to a 
fan $\Sigma\subset N_{\qq}$ with torsion free class group, 
let $\iota:N\to N'$ be a lattice monomorphism with finite cokernel $G$ 
and  $\pi:X\to X'$ be the associated finite quotient.
If the primitive generators $n_1,\dots,n_r\in N$  
of the rays of the fan of $X$ are primitive in $N'$,
then the homomorphism $\pi^*:\R(X')\to \R(X)$ 
can be taken to be the identity and   
$\Cl(X')\cong \Cl(X)\oplus G$.
\end{lemma}

\begin{proof} 
Let $P,P'$ be the $P$-morphisms of $X,X'$ respectively.
By the primitivity assumption on the $n_i$'s in $N'$,
we have $P'=\iota\circ P$.
Thus we have the following commutative diagram with exact rows
\[
\xymatrix{
0\ar[r] & M'\ar[r]^{{P'}^T}\ar[d]^{\iota^{\vee}} & \zz^r\ar[r]\ar[d]^{id} & K'\ar[r]\ar[d]^{pr_1}& 0\\
0\ar[r] & M\ar[r]^{P^T} & \zz^r\ar[r] & K\ar[r]& 0\\
}
.\]
Observe that the central vertical arrow describes 
the map $\pi^*$ in Cox coordinates. 
Since $K$ is free, then $K'$ is isomorphic to $K\oplus \ker(pr_1)$.
Chasing in the diagram one finds that $\ker(pr_1)$ is isomorphic to 
$M/M'\cong G$.
\end{proof}

 \subsection{Polytopes}
We recall that a lattice polytope $\Delta$ in $M_{\qq}$ containing the origin in its interior is 
\begin{itemize}
\item {\em reflexive} if $\Delta^*$ is a lattice polytope;
\item {\em canonical} if the origin is its unique interior lattice point;
\item $\qq${-\em Fano} if its vertices are primitive in $M$.
\end{itemize}

A reflexive polytope can be defined equivalently as 
a lattice polytope with the origin in its interior such that 
the integral distance between any of its facets 
and the origin is equal to one (\cite[Theorem 4.1.6]{Ba}). 
Clearly with our definition $\Delta$ is reflexive 
if and only if $\Delta^*$ is reflexive.
A canonical polytope is clearly $\qq$-Fano 
since, given a non-primitive vertex $mv$, $m\in \zz_{>0}$, 
the vector $v$ would be a non zero interior lattice point.
Thus we have the following implications:
\[
\text{reflexive }  \Rightarrow  \text{ canonical }\Rightarrow\ \qq\text{-Fano}.
\]

\begin{remark} 
In \cite{Ka} Kasprzyk provided a classification of 
three dimensional canonical polytopes, 
available in the Graded Ring Database  
 \href{http://www.grdb.co.uk/}{http://www.grdb.co.uk/},
and described an approach to the classification 
in higher dimension.
Polytopes with the origin in their interior, 
sometimes called polytopes with the IP property in the literature, 
had a key role in the classification of reflexive polytopes of dimension $\leq 4$ 
 by Kreuzer and Skarke \cite{Kreuzer1998, Kreuzer2000, Kreuzer2002}.
\end{remark}

We recall that a projective normal variety $X$ 
is $\qq$-{\em Fano} if  $-K_X$ is $\qq$-Cartier and ample 
(in particular it is $\qq$-Gorenstein) and 
{\em Fano} if moreover $-K_X$ is Cartier (in particular it is Gorenstein).
Moreover, the following holds (see \cite[Theorem 6.2.1, Proposition 11.4.12, Theorem 8.3.4]{CLS}).

\begin{theorem}\label{qfano} Let $\Delta\subset N_\qq$ be a  lattice polytope containing 
the origin in its interior. Then $X_{\Delta^*}$  is
\begin{enumerate}[$\bullet$]
\item $\qq$-Fano  if and only if $\Delta$ is  $\qq$-Fano;
\item $\qq$-Fano with canonical singularities if and only if $\Delta$ is canonical;
\item Fano if and only if $\Delta$ (or $\Delta^*$) is reflexive.
\end{enumerate}
\end{theorem}

We finally introduce the concept of good pair, a key word in the paper.

\begin{definition} Let $\Delta_1,\Delta_2$ be two polytopes in $M_\qq$.
We will say that $(\Delta_1,\Delta_2)$ is a {\em good pair} 
if $\Delta_1\subseteq \Delta_2$, and $\Delta_1, \Delta_2^*$ are canonical
(in particular $\Delta_1$ and $\Delta_2^*$ are both lattice polytopes). 
\end{definition}

The following shows that a lattice polytope containing the origin 
in its interior is canonical as soon as its dual is big enough.

\begin{lemma}\label{canon1}
Let $\Delta \subset M_\qq$ be a lattice polytope containing the origin as an interior point.
Then the origin is the only lattice interior point of $\Delta^*$.
\end{lemma}

\begin{proof}
By definition of the polar polytope it is clear that it contains the origin in its interior.
Suppose now that $\Delta^*$ contains a lattice point $n\neq 0$ in its interior. 
In particular for $\varepsilon > 0$ small enough $(1+\varepsilon) n$ is contained in $\Delta^*$. 
By definition of the polar polytope, this implies that   
$
\langle u, (1+\varepsilon)n\rangle \geq -1
$ 
for all $u \in \Delta$.
Thus every lattice point $u$ in $\Delta$ satisfies 
\[
\langle u, n\rangle \geq \dfrac{-1}{(1+\varepsilon)} > -1
\]
and hence
$\langle u, n\rangle \geq 0$ since $\langle u, n\rangle$ is an integer.
It follows that the lattice polytope $\Delta$ is 
contained in the half space $H_n=\lbrace u \in M_\rr \mid \langle u, n\rangle \geq 0 \rbrace$, 
contradicting the fact that $\Delta$ contains the origin as an interior point.
\end{proof}

As an immediate consequence we have:

\begin{corollary}\label{canon2}
Let $\Delta_1$ and $\Delta_2$ two polytopes in $M_\qq$ with $\Delta_1 \subseteq \Delta_2$. Then $(\Delta_1, \Delta_2)$ is a good pair if and only if $\Delta_1$ and $\Delta_2^*$ are lattice polytopes containing the origin as an interior point.
\end{corollary}

This also implies the following result,
explaining how good pairs behave 
when changing the underlying lattice up to finite index.

\begin{lemma}\label{gp}
Let $(\Delta_1,\Delta_2)$ be a good pair in $M_\qq$ and let 
$M'\subset M$ be an inclusion of lattices with finite index. 
If the vertices of $\Delta_1$ belong to $M'$,
then $(\Delta_1,\Delta_2)$ is a good pair in $M'_{\qq}$.
\end{lemma}

\section{Anticanonical hypersurfaces}\label{sec-hyp}
Let $X$ be a projective toric variety defined by a fan $\Sigma\subset N_\qq$ 
and let $n_1,\dots, n_r\in N$ be
the primitive generators of the one dimensional cones of  $\Sigma$.
The {\em anticanonical polytope} of $X$ is the polytope
\[
\Theta=\{m\in M_\qq: \langle m,n_i\rangle\geq -1,\ \forall\, i\}.
\]
The lattice points of $\Theta$ naturally give a basis 
for the Riemann-Roch space of the divisor $-K_X=\sum_iD_i$. 
In fact, given $u\in \Theta \cap M$, 
the vector $P^T(u)+{\bf 1}\in \zz^r$, 
where ${\bf 1}$ is the vector with all entries equal to $1$,
is the vector of  exponents of a monomial $m_u$ 
in the Cox ring $\mathcal R(X)$ of degree $[-K_X]$.
Conversely, any such monomial can be obtained in the same way.

Given a hypersurface $D$ of $X$  of degree $[-K_X]$ 
defined by $f=0$ in Cox coordinates,
we define the {\em support} ${\rm supp}(f)$ 
to be the set of $u\in M$ such that $m_u$ is a 
monomial of $f$.
Moreover we define the {\em Newton polytope} of 
 $D$ as the convex hull of the points in its support.
Given a lattice polytope $\Delta$ contained in $\Theta$ 
we will denote by $\mathcal F_{\Delta,\Theta^*}$ the family of anticanonical 
hypersurfaces of $X$ whose Newton polytope is equal to $\Delta$.

\subsection{Regularity of hypersurfaces}
In this section we will translate some basic regularity 
properties of hypersurfaces in $\mathcal F_{\Delta,\Theta^*}$ 
in terms of geometric properties of $\Delta$. 
 We recall that a hypersurface $D$ of a projective toric variety $X$ is called
{\em well-formed} if 
\[
{\rm codim}_D(D\cap {\rm Sing}(X))\geq 2,
\]
 where ${\rm Sing}(X)$  is 
the singular locus of $X$.

\begin{example}
In case $X$ is a normalized weighted projective space, i.e.  
$X=\pp(w_1,\dots,w_n)$ with ${\rm gcd}(w_1,\dots,\hat w_i,\dots, w_n)=1$,
it is known \cite{F} that the general anticanonical hypersurface 
is well-formed if and only if 
\[
{\rm gcd}(w_1,\dots,\hat w_i,\dots, \hat w_j,\dots w_n)| \sum_kw_k.
\]
\end{example}

We will need the following result, where 
$D_I:=\cap_{i\in I} D_i$ for $I\subseteq \{1,\dots, r\}$.

\begin{lemma}\label{hyp} Let  $X$ be a $\qq$-Fano toric variety 
with canonical singularities and let $\Delta$ be a lattice polytope 
contained in its anticanonical polytope.
Then $D_I$ is not empty if and only if $\{n_i: i\in I\}$ 
is contained in a facet of $\Theta^*$ and 
a hypersurface in $\mathcal F_{\Delta,\Theta^*}$ contains $D_I$
if and only if  $\{n_i: i\in I\}$ is not contained in a facet of $\Delta^*$. 
\end{lemma}
\begin{proof}
Let  $\Theta$ be the anticanonical polytope of $X$.
By the assumption on $X$, a fan $\Sigma$ 
for $X$ is given by the cones over the facets of $\Theta^*$ 
and the $n_i$ are the vertices of $\Theta^*$.
The stratum $D_I$ is not empty if and only if the 
set $\{n_i: i\in I\}$ is contained in a cone of $\Sigma$,
or equivalently if the $n_i$ are  contained in a facet of $\Theta^*$.
This gives the first statement.

Let $u\in \Delta\cap M$, 
and $m_u$ be the corresponding monomial in homogeneous 
coordinates.
The zero set of the monomial $m_u$ is given by
\[
{\rm div}(m_u)=\sum_{i=1}^r(\langle u,n_i\rangle +1)D_i.
\] 
A hypersurface $f=0$ in $\mathcal F_{\Delta,\Theta^*}$
contains $D_I$ if and only if 
any monomial $m_u$ in $f$ vanishes 
along some of the $D_i$'s with $i\in I$.
This means that for all $u\in {\rm supp}(f)$ 
there exists some $i\in I$ with 
$\langle u, n_i\rangle > -1$.
In particular this holds for the vertices of 
$\Delta$, which means that the $n_i$'s do not 
all belong to a single facet of $\Delta^*$.
\end{proof}

\begin{remark}\label{genhyp} In the following results we will usually ask 
$D$ to be a {\em general element} in $\mathcal F_{\Delta,\Theta^*}$.
This means that $D=\{f=0\}$, where the Newton polytope of $D$  
is $\Delta$ and the coefficients of $f$ are general. 
Observe that the support of $f$ is not necessarily equal to $\Delta\cap M$.
\end{remark}

\begin{proposition}\label{norm}
Let $X$ be a $\qq$-Fano toric variety with canonical singularities 
and let $\Delta\subset M_\qq$ be a lattice polytope contained in the 
anticanonical polytope $\Theta$ of $X$.
A general hypersurface in $\mathcal F_{\Delta,\Theta^*}$ is: 
\begin{enumerate}[i)]
\item irreducible if and only if $n_i$ belongs to the boundary of $\Delta^*$ for any $i$;
\item well-formed  if and only if, anytime $n_i,n_j$ belong to a facet of $\Theta^*$
and not to a facet of $\Delta^*$, the segment joining them doesn't contain any lattice point;
\item normal if, anytime $n_i,n_j$ belong to a facet of $\Theta^*$
and not to a facet of $\Delta^*$,  $n_i+n_j$ is not in the interior of $\Delta^*$.
\end{enumerate}
\end{proposition}

\begin{proof}
Since $X$ is $\qq$-Fano with canonical singularities, 
the $n_i$'s are the vertices of $\Theta^*$ 
and the origin is the only interior lattice point of both $\Theta$ and $\Theta^*$.
In what follows $D$ denotes a  general element in $\mathcal F_{\Delta,\Theta^*}$.

By the second Bertini's theorem $D$ is reducible if and only if 
it contains one of the integral invariant 
divisors $D_i$ for the torus action as a component.
Thus i) follows from Lemma \ref{hyp}.

By the same Lemma, $D$ contains  
the stratum $D_{ij}$ if and only if 
$n_i,n_j$ are contained in a facet of 
$\Theta^*$ and not in a facet of $\Delta^*$.
Moreover, $X$ is singular along $D_{ij}$ if and only 
if the triangle $0,n_i,n_j$ contains a lattice point $n$
outside its vertices. 
Since the only interior lattice point of $\Theta^*$ 
is the origin, this means that  $n$ belongs to the 
segment between $n_i,n_j$.
This gives ii).

Let $p:\hat X\to X$ be the characteristic space of $X$,  
let $\hat D=p^{-1}(D)$ and let $\hat D_i=p^{-1}(D_i)$.
By Serre's criterion \cite[Proposition 8.23, Ch.II]{Ha} 
$\hat D$ is normal if and only 
if it is smooth in codimension one.
By the first Bertini's theorem this happens if and only if 
$\hat D_{ij}:=\hat D_i\cap \hat D_j$ 
is not contained in the singular locus of $\hat D$,
whenever it is not empty. 
By Lemma \ref{hyp}  $\hat D_{ij}$ is not empty 
and it is contained in $\hat D$
when $n_i,n_j$ belong to the same facet of $\Theta^*$  
but not to a facet of $\Delta^*$.
Under these conditions, $\hat D$ 
is singular along $\hat D_{ij}$ if and only if, 
for any $u\in {\rm supp}(f)$,
$(u,n_j)>-1$ whenever $(u, n_i)=0$, and similarly 
changing the role of $i$ and $j$ (this 
is equivalent to ask that the partial derivatives of the equation of $\hat D$ 
in homogeneous coordinates vanish along $\hat D_{ij}$).
Since there exists no $u\in \Delta\cap M$ 
such that $(u,n_i)=(u,n_j)=-1$, 
this is equivalent to ask that 
$(u, n_i+n_j)>-1$ for all $u\in \Delta\cap M$, i.e. 
that $n_i+n_j$ belongs to the interior of $\Delta^*$.

We recall that $p$ is a GIT quotient for the action of 
a quasi-torus $T$. 
The divisor $\hat D$ is $T$-invariant, being defined by 
a homogeneous polynomial in $\mathcal R(X)$.
This implies that $p|_{\hat D}: \hat D\to D$ is still a GIT quotient 
for the action of the group $T/T_0$, where $T_0$ 
is the subgroup of $T$ acting trivially on $\hat D$.
Since $\hat D$ is normal, it follows that $D$ is normal 
(see for example~\cite[Lemma 5.0.4]{CLS}). 
This proves iii).
\end{proof}

\subsection{Hypersurfaces with canonical singularities}
 
Let $X$ be a $\qq$-Gorenstein 
normal variety over $\cc$ 
of dimension $\geq 2$ 
and let $D$ be a $\qq$-Cartier divisor on $X$. 
Given a resolution $f:\tilde X\to X$,
that is a proper birational 
morphism such that $\tilde X$ is smooth,
one can write
\[
(K_{\tilde X}+f_*^{-1}D)-f^*(K_X+D)\equiv \sum_{i=1}^r a(X,D,E_i)E_i,
\]
where $E_1,\dots E_r$ are the distinct irreducible 
components of the exceptional divisor of $f$ and 
$f_*^{-1}D$ denotes the proper birational transform of $D$ 
(see \cite[Remark 6.6]{KSC} for the precise meaning of this equation). 
 The {\em discrepancy of } the pair $(X,D)$, denoted by ${\rm discrep}(X,D)$, 
  is the infimum of the values $a(X,D,E)$, 
  as $E$ varies over all exceptional divisors of the resolutions of $X$.
 The pair $(X,D)$ is {\em canonical} if ${\rm discrep}(X,D)\geq 0$ and
 $X$ has {\em canonical singularities} if $(X,0)$ is canonical.

In order to compute the discrepancy of a pair $(X,D)$ 
it is enough to consider the minimum over the values
$a(X,D,E)$ as $E$ varies among the exceptional divisors 
of a given {\em log resolution} of $(X,D)$,
i.e. a resolution  such that  
${\rm Exc}(f)+f_*^{-1}(D)$ has pure codimension 1 and is 
a divisor with simple normal crossings \cite[Definition 6.21]{KSC}.
Such a resolution always exists by a theorem of Hironaka \cite[Theorem 10.45]{KK}.
We recall the following result, which relates the discrepancy 
of a pair $(X,D)$ to the discrepancy of $D$.

\begin{theorem}[\cite{K92}]\label{kollar}
Let $X$ be a normal variety over $\cc$ 
and let $D$ be a normal divisor on $X$ 
such that $K_X+D$ is $\qq$-Cartier and 
${\rm codim}_D({\rm Sing}(X)\cap D)\geq 2$. 
Then 
\begin{equation}\label{discr}
{\rm discrep}(D)\geq {\rm discrep}({\rm center}\subset D, X,D),
\end{equation}
where the right hand side is the infimum of the values
 $a(X,D,E)$, where $f(E)\subset Z$.
In particular $D$ has canonical singularities if the pair $(X,D)$ 
is canonical.
\end{theorem}

\begin{proof}
The inequality follows from \cite[Proposition 5.46]{KM} 
taking $Z=S$ and $B=0$ or \cite[\S 17.2]{K92}.
The last statement immediately follows 
since ${\rm discrep}({\rm center}\subset D, X,D) \geq  {\rm discrep}(X,D)$.
 \end{proof}

\begin{proposition}\label{can}
Let $X$ be a $\qq$-Fano toric variety with canonical singularities 
and let
$\Delta\subset M_\qq$ be a lattice polytope contained in the 
anticanonical polytope $\Theta$ of $X$.
If $\Delta$ is a canonical polytope 
then the general element $D$ of $\mathcal F_{\Delta,\Theta^*}$ 
is well-formed, normal and has canonical singularities. 
\end{proposition}

\begin{proof}
If $\Delta$ is canonical, then $\Delta^*$ 
is a polytope and its only interior lattice point is the origin 
by  Lemma \ref{canon1}. 
By Proposition \ref{norm} we have that $D$ is well-formed 
since, if $n_i,n_j$ belong to a facet of $\Theta^*$ and not to a facet 
of $\Delta^*$,  then the segment joining them 
intersects the boundary of $\Delta^*$ only at $n_i,n_j$.
Moreover, by the same proposition, $D$ is normal.

Since $D$ is general, there exists a toric resolution 
of singularities $f:\tilde X\to X$  
which is a log resolution for $D$, 
obtained by means of a refinement $\tilde\Sigma$ of 
the fan $\Sigma$ of $X$. 
This can be obtained taking first a toric resolution 
of the singularities of $X$ and then successive 
toric blow-ups along the base locus of $\mathcal F_{\Delta,\Theta^*}$ 
until its proper transform is base point free.
By the first Bertini's theorem the general element $\tilde D$ 
of such proper transform is smooth.
Moreover, the same theorem implies that $\tilde D$
intersects transversally each component 
of the exceptional locus, since its restriction 
to any such component is base point free.

Let $E$ be an exceptional divisor of $f$ 
and let $n\in N$ be the primitive generator of the 
corresponding ray of $\tilde \Sigma$.
Observe that 
\[
{\rm mult}_{E}(K_{\tilde X}-f^*(K_X))=-1-{\rm mult}_{E}f^*(K_X)
\] 
since $E$ is one of the integral torus invariant divisors of $\tilde X$.
We can write $D={\rm div}(\chi)-K_X$,
where  $\chi$ is a general linear combination of 
$\chi^{u}$, where $u$ belongs to the support ${\rm supp}(\phi)$ 
of a defining equation $\phi$ of $D$.
Then 
\[
{\rm mult}_{E}(f^*(D)-f_*^{-1}(D))={\rm mult}_{E}(f^*(D))
={\rm mult}_{E}(f^*{\rm div}(\chi))-{\rm mult}_{E}f^*(K_X)
\]
\[
=\min_{u\in {\rm supp}(\phi)}\{{\rm mult}_{E}(f^*\chi^{u})\}-{\rm mult}_{E}f^*(K_X)
=\min_{u\in {\rm supp}(\phi)}\{(u,n)\}-{\rm mult}_{E}f^*(K_X),
\]
where the third equality is due to the generality assumption on 
$D$  and the last equality to the fact that
\[
{\rm mult}_Ef^*{\rm div}(\chi^{u})={\rm mult}_E{\rm div}(f^*\chi^{u})=(u,n).
\]
This gives
\begin{equation}\label{discr}
a(X,D,E)=-1-\min_{u\in {\rm supp}(\phi)}\{(u,n)\}=-1-\min_{u\in \Delta\cap M}\{(u,n)\},
\end{equation}
where the second equality is due to the fact that ${\rm supp}(\phi)$ 
contains the vertices of $\Delta$.
Such discrepancy is non-negative 
since $\Delta^*$ has no non-zero interior lattice point.
Theorem \ref{kollar} thus implies that 
$D$ has canonical singularities.
\end{proof}

\begin{corollary}\label{crep}
Let $X$ be a $\qq$-Fano toric variety with canonical singularities,
$\Delta$ be a canonical polytope contained in its anticanonical polytope  $\Theta$
and $D$ be a general element in $\mathcal F_{\Delta,\Theta^*}$.
A birational toric morphism $f:\tilde X\to X$ induces a crepant morphism 
$f^*(D)\to D$ if and only if the rays of the fan of $\tilde X$ 
are generated by nonzero lattice points in $\Delta^*$.
\end{corollary}
\begin{proof}
It follows from (\ref{discr}) that $a(X,D,E)=0$ if and only if 
$n\in \Delta^*$.
\end{proof}

\begin{remark}
By the proof of Proposition \ref{can} the discrepancies 
of a general anticanonical hypersurface $D$ 
only depend on its Newton polytope $\Delta$.
\end{remark}

\begin{remark}\label{conj}
As a consequence of the adjunction Conjecture \cite[Theorem 4.9]{KK} 
formulated by Shokurov and Koll\'ar,
the inequality (\ref{discr}) is actually an equality.
We now show that, under such conjecture, 
the condition on the polytope $\Delta$ 
in Proposition \ref{can} is also a necessary condition for
$D$ to be normal with canonical singularities.
Assume that $n\in N$ 
is a non-zero primitive vector in the interior of  $\Delta^*$.
Let $\tilde \Sigma$ be a smooth fan refining the 
star subdivision of the fan $\Sigma$ of $X$ induced by $n$.
This gives a resolution $f$ of $X$ 
and $n$ corresponds to an exceptional divisor $E$ of $f$.
Let $\sigma$ be the cone of $\Sigma$ containing 
$n$ in its interior.  
The primitive generators of the 
rays of $\sigma$ are not contained 
in a facet of $\Delta^*$, 
since otherwise this would also be a facet 
of $\Theta^*$ and $n$ would be an interior point 
of $\Theta^*$, contradicting the fact 
that $X$ has canonical singularities (see Theorem \ref{qfano}).
Thus $f(E)\subseteq D$ by Lemma \ref{hyp}.
Since $n$ is in the interior of $\Delta^*$, the  
computation in the proof of Proposition \ref{can} 
gives that $a(X,D,E)<0$,
thus by \cite[Theorem 4.9]{KK}  
$D$ has a non-canonical singularity.
\end{remark}

\begin{corollary}\label{gencor}
Let $X$ be a $\qq$-Fano toric variety with canonical singularities.
If the convex hull of the lattice points of the anticanonical polytope of $X$  
is a canonical polytope, then a general anticanonical hypersurface of $X$ 
 is well-formed, normal and 
has canonical singularities.
\end{corollary}
\begin{proof}
It follows from Proposition \ref{can} taking $\Delta$ to be the 
the convex hull of $\Theta\cap~ M$.
\end{proof}

\subsection{Quasismooth hypersurfaces}
A hypersurface $D$ of a projective toric variety $X$ is 
called {\em quasismooth (or transverse)} if $p^{-1}(D)$ is smooth, 
where $p:\hat X\to X$ 
is the quotient map in the Cox construction of $X$.

If $X$ is a weighted projective space, 
a quasismooth hypersurface of $X$ of dimension $\geq 3$ 
is known to be well-formed, 
unless it is isomorphic to a toric stratum \cite[Proposition 6]{Dim}.
This result can be generalized as follows.

\begin{proposition}
Let $X$ be a projective toric variety  
whose irrelevant locus has codimension 
$>4$ in $\hat X$.
A quasismooth hypersurface $D$ of  
 $X$ is either well-formed or it 
 is isomorphic to a toric stratum of $X$.
\end{proposition}

\begin{proof}
Let $f$ be a defining element for $D$ 
in the Cox ring $R(X)=\cc[x_1,\dots,x_r]$.
Assume that $D$ is not well-formed, 
in particular it contains a codimension two 
toric stratum of $X$.
Thus we can assume $f$ to be of the form 
\[
f(x_1,\dots,x_r)=x_1f_1+x_2f_2.
\]
Computing the partial derivatives of $f$ 
one can see that they all vanish 
along the subset $S$ of $\hat X$ defined by 
$\{x_1=x_2=f_1=f_2=0\}$.
If neither $f_1$ or $f_2$ is constant,
we have that 
\[
\dim(S)\geq \dim(\hat X)-4>\dim(\bar X-\hat X),
\] 
contradicting the fact that $D$ is quasismooth.
Thus we can assume that $f_1$ is constant,
so that $f(x_1,\dots, x_r)=\alpha x_1+x_2f_2$ 
is isomorphic to the stratum $x_1=0$ by the isomorphism 
$(x_2,\dots, x_r)\mapsto (-x_2\alpha^{-1}f_2,x_2,\dots,x_r)$.
\end{proof}

Quasismooth and well-formed anticanonical 
hypersurfaces give a class of hypersurfaces 
with canonical singularities. However, as we will 
observe later, such class is quite small 
in dimension bigger than three.

\begin{proposition}\label{prop-canonicalsing}
Let $D$ be an anticanonical hypersurface 
of a projective toric variety $X$.
If $D$ is quasismooth and well-formed, 
then $D$ has canonical singularities.
\end{proposition}
\begin{proof}
Since $D$ is quasismooth, then $D$ is normal by 
the proof of Proposition \ref{norm}.
Moreover, by \cite[Proposition 4.5, (1) and (5)]{KK}  
the adjunction formula holds for $D$ 
and gives that $K_D\sim \Osh_D$.
In particular $D$ has Gorenstein singularities.
Moreover, since $\hat D$ is smooth, the singularities 
of $D$ are rational \cite[Corollaire]{Bo}.
By \cite[Corollary 11.13]{Ko2} Gorenstein rational 
singularities are canonical.
\end{proof}

\begin{remark}
In \cite{Ba} Batyrev considered 
a different notion of regularity: an anticanonical hypersurface 
$Y$ of a toric Fano variety $X$ is regular if 
the intersection of $Y$ with 
any toric stratum of $X$ is either empty 
or smooth of codimension one.
Any regular hypersurface is quasismooth,
see \cite[Proposition 4.15]{BC} or \cite[Proposition 5.3]{C}.
\end{remark}

\subsection{Examples}\label{examples}
Given a $\qq$-Fano toric variety with canonical 
singularities and with anticanonical polytope $\Theta$, 
we can consider three different 
properties for the convex hull $\bar\Theta$ 
of the lattice points of $\Theta$: 
canonical, reflexive and quasismooth. 
Here we say that $\bar\Theta$ is {\em quasismooth} if 
such property holds for the general anticanonical hypersurface.
In the case of weighted projective spaces 
it is known that the three concepts are equivalent 
in dimension $2$ and $3$.
In dimension $4$, canonical implies reflexive \cite[Theorem, \S 3]{Sk}.
Moreover, for weighted projective spaces of any dimension, 
quasismooth  implies canonical \cite[Lemma 2]{Sk}.
In higher dimension the concepts of reflexive and quasismooth 
are unrelated and there are examples of canonical
polytopes which are neither quasismooth nor reflexive.
In Table \ref{table-conta} we show
the number of weight systems $w=(w_1,\ldots,w_6)$ 
with $w_i\leq 10$
such that the anticanonical polytope $\Theta$ of $\pp(w)$ 
is reflexive ($F$), 
 $\bar\Theta$  is reflexive ($R$),
 $\bar\Theta$ is canonical ($C$) and not reflexive 
 and we distinguish whether the general anticanonical 
hypersurface of $\pp(w)$ is quasismooth ($Q$) or not. 
The properties of the anticanonical 
polytope can be checked 
by means of Magma \cite{Magma} 
and the programs available here:
\begin{center}
\href{http://goo.gl/A7W17Z}{http://goo.gl/A7W17Z}.
\end{center}
\noindent See also the Calabi-Yau data webpage by  Kreuzer and Skarke 
\begin{center}
\href{http://hep.itp.tuwien.ac.at/~kreuzer/CY/}{http://hep.itp.tuwien.ac.at/~kreuzer/CY/}.
\end{center}

\begin{table}[h!]
\begin{tabular}{c|c|c|c|c}
weights up to& $F$ & $R$ & $Q$ and not $R$ & $C$ not $R$ and not $Q$\\
\hline
2&3&4&1&0\\
3&6&13&5&2\\
4&10&39&11&3\\
5&15&83&30&30\\
6&28&164&45&63\\
7&31&300&89&193\\
8&44&524&133&358\\
9&52 & 833 & 190 & 747\\
10&71& 1278& 269 & 1221
\end{tabular}
\vspace{0.4cm}

\caption{Counting weight systems in dimension $5$}
\label{table-conta}
\end{table}

\begin{remark}
Let $X$ be a $\qq$-Fano toric variety and assume that 
there exists a toric Fano variety $X'$ and a 
birational toric map $X\to X'$ which induces 
a bijection between the anticanonical linear series $|-K_X|$ 
and $|-K_{X'}|$. By standard facts in toric geometry, 
this map is induced by an isomorphism $\varphi:M\to M'$ 
which gives a bijection between the lattice points 
of the anticanonical polytope $\Theta$ of $X$ and those 
of the anticanonical polytope $\Theta'$ of $X'$. 
In particular $\varphi_\qq$ induces an isomorphism 
between the convex hulls of the lattice points of $\Theta$ 
and $\Theta'$.
Since $X'$ is Fano, the polytope $\Theta'$ is reflexive by \cite[Theorem 4.1.9]{Ba},
thus the convex hull of the lattice points of $\Theta$ is reflexive.
This shows that, if $\bar\Theta$ is not reflexive,
then the general anticanonical hypersurface of $X$ 
is not (torically) birational to a hypersurface in a toric Fano variety. 
\end{remark}

\begin{example}[$R$, $C$ and $Q$]
In dimension 3 there are 104 weighted projective spaces with canonical singularities; 
for 95 of them the convex hull of lattice points of the anticanonical polytope is a canonical, reflexive and quasismooth polytope.
Moreover, for 14 of these weight systems the anticanonical polytope is reflexive, i.e. the weighted projective space is Fano.
\end{example}

\begin{example}[$R$ and not $Q$]\label{r}
$X=\pp(1,1,1,3,4) $
is a toric $\qq$-Fano variety such that $\bar\Theta$ is reflexive.
Observe that an anticanonical hypersurface of $X$ 
is defined by an equation of the form:
\[
f(x_1,\dots,x_5)=x_1f_1+x_2f_2+x_3f_3+x_4f_4,
\]
since it has degree $10$ and there is no power 
of such degree in the variable $x_5$.
All partial derivatives vanish at the point $(0:0:0:0:1)$ since 
$f_1,f_2,f_3,f_4$ do not contain a power 
of $x_5$. Thus the general anticanonical 
hypersurface of $X$ is not quasismooth. 
\end{example}

\begin{example}[$Q$ and not $R$]\label{q}
$X=\pp(1,1,1,1,1,2) $
is a  toric $\qq$-Fano variety such that $\bar\Theta$ is canonical and
not reflexive. 
A general anticanonical hypersurface of $X$ 
is defined by an equation of the form
\[
f(x_1,\dots,x_6)=x_1f_1+x_2f_2+x_3f_3+x_4f_4+x_5f_5.
\]
By Bertini's theorem, the singular locus of $f$ in $\cc^6$
is contained in the base locus of the corresponding 
linear system, which only contains the point $(0:0:0:0:0:1)$. 
The partial derivatives of  $f$ do not vanish at such point since 
we can assume that $f_1$ (for example) 
contains the monomial $x_6^3$. Thus $f=0$ is quasismooth. 
More examples are given in Table \ref{tableqs}.
\end{example}

\begin{example}[$C$ not $R$ and not $Q$] \label{c}
$X=\pp(1,1,2,3,3,3)$ 
is a  toric $\qq$-Fano variety such that $\bar\Theta$ 
is canonical and not reflexive and such that the general anticanonical 
hypersurface is not quasismooth. 
More examples are given in Table \ref{tableqs}.
\end{example}

 \begin{table}[h!]
\footnotesize
\begin{tabular}{c|c}
$C$, not $R$, $Q$&$C$, not $R$, not $Q$\\
\hline

    ( 1, 1, 1, 1, 1, 2 )&( 1, 1, 2, 3, 3, 3 )\\
    ( 1, 2, 2, 2, 3, 4 )&( 1, 1, 2, 3, 3, 4 )\\
    ( 1, 1, 1, 1, 2, 3 )&( 1, 1, 1, 2, 3, 3 )\\
    ( 1, 1, 2, 2, 2, 3 )&\\
    ( 1, 2, 3, 3, 3, 3 )&\\
    ( 1, 1, 1, 3, 3, 4 )&\\
    ( 1, 2, 2, 3, 3, 4 )&\\
    ( 1, 1, 2, 2, 3, 4 )&\\
    ( 1, 1, 3, 3, 3, 4 )&\\
    ( 1, 2, 2, 3, 4, 4 )&\\
    ( 1, 1, 1, 2, 2, 3 )&\\
    
\end{tabular}
\vspace{0.4cm}
\caption{Weights $\leq 4$ having quasismooth and non-quasismooth general anticanonical hypersurface}
\label{tableqs}
    \end{table}

\section{A duality between families of Calabi-Yau hypersurfaces}\label{sec-duality}
 We recall that an $n$-dimensional normal projective variety $Y$ is a {\em Calabi-Yau variety} 
if it has canonical singularities, $K_Y\cong \Osh_Y$ and $h^i(Y,\Osh_Y)=0$ 
for $0<i<n$.

In \cite[Theorem 4.1.9]{Ba} Batyrev proved that a projective toric variety 
$X_{\Delta}$ is Fano, or equivalently its anticanonical polytope $\Delta$ is reflexive, 
if and only if regular anticanonical hypersurfaces $Y$ of $X$ 
are Calabi-Yau varieties. 
Moreover he defines a duality between families 
of anticanonical hypersurfaces of Fano toric varieties:
\[
\mathcal F_{\Delta,\Delta^*}\subseteq |-K_{X_{\Delta}}| \longleftrightarrow \mathcal F_{\Delta^*,\Delta}\subseteq |-K_{X_{\Delta^*}}|. 
\]
In this section we will introduce a generalization 
of this duality in case $X$ is $\qq$-Fano 
and the family of hypersurfaces is not necessarily the 
full anticanonical linear system.
Such generalization is based on the result given in Theorem \ref{cy};
using the characterization of Proposition \ref{can} we can now prove it.  
Observe that by Remark \ref{conj}, if the equality holds in \eqref{discr},
this would provide a characterization of 
$\qq$-Fano toric varieties whose general 
anticanonical hypersurfaces are Calabi-Yau.

\begin{proof}[Proof of Theorem \ref{cy}]
In what follows $D$ will denote a 
general anticanonical hypersurface of $\mathcal F_{\Delta,\Theta^*}$.
By Proposition \ref{can} $D$ is well-formed, normal 
and has canonical singularities.
By \cite[Proposition 4.5, (1) and (5)]{KK}
the adjunction formula holds for $D$,
giving that $K_D$ is trivial.
Moreover we have the exact sequence
 \[
 \xymatrix{
 0\ar[r] & \Osh_{X}(K_X)\ar[r] & \Osh_X\ar[r] & \Osh_D\ar[r]& 0,
 }
 \]
 which induces the exact sequence
 \[
 \xymatrix{
 \dots\ar[r] & H^i(X,\Osh_X)\ar[r] & H^i(D,\Osh_D)\ar[r] &H^{i+1}(X,\Osh_X(K_X))\ar[r] &\dots. 
 }
 \]
Since $h^i(X,\Osh_X)=0$ for $i>0$ and 
$h^{i+1}(X,\Osh_X(K_X))=h^{\dim(X)-i-1}(X,\Osh_X)$ for $i<\dim(X)-1$ 
by Serre-Grothendieck duality, we obtain the  
vanishing of $h^i(D,\Osh_D)$ for $0<i<\dim(D)$.
Thus $D$ is a Calabi-Yau variety.
\end{proof}

\begin{remark} In \cite[Theorem 2.25]{MP} the author states that 
 a general anticanonical hypersurface of a projective toric variety
  is a Calabi-Yau variety 
 if and only if the convex hull of the lattice points of its 
 anticanonical polytope is reflexive.
 This is not true in general, see Example \ref{q}.
\end{remark}

 Observe that a  good pair of polytopes 
 $\Delta_1\subset \Delta_2$  naturally produces a family of Calabi-Yau varieties 
in a $\qq$-Fano projective toric variety.
In fact, the toric variety $X:=X_{\Delta_2}$ defined by the normal fan to $\Delta_2$ 
is $\qq$-Fano with canonical singularities by Theorem \ref{qfano} and
$\Delta_2$ is its anticanonical polytope.
The subpolytope $\Delta_1\subseteq \Delta_2$ 
identifies the  family $\mathcal F_{\Delta_1,\Delta_2^*}$ of anticanonical 
hypersurfaces of $X$, whose general element is a Calabi-Yau variety 
by Proposition \ref{can}.
 By our definition of good pair we immediately have that if
$(\Delta_1,\Delta_2)$  is a good pair in $M_\qq$, then its
polar $(\Delta_2^*,\Delta_1^*)$ is a good pair in $N_\qq$.
This provides a duality between families of Calabi-Yau 
hypersurfaces of $\qq$-Fano toric varieties:
\[
\mathcal F_{\Delta_1,\Delta_2^*}\subseteq |-K_{X_{\Delta_2}}| \longleftrightarrow \mathcal F_{\Delta_2^*,\Delta_1}\subseteq |-K_{X_{\Delta_1^*}}|. 
\]

\begin{proposition}
If $\Delta_1=\Delta_2$, then the duality between 
good pairs is Batyrev duality.
\end{proposition}
\begin{proof}
 If $\Delta_1=\Delta_2$, 
then $\Delta_2$  and $\Delta_2^*$ 
are lattice polytopes, thus $\Delta_2$ is reflexive.
By Theorem \ref{qfano} this means that $X_{\Delta_2}$ 
is a Fano variety. Moreover $\mathcal F_{\Delta_2,\Delta_2^*}$ 
is the family of all anticanonical hypersurfaces of $X_{\Delta_2}$.
\end{proof}

\begin{example} Let $X=\pp^2$ and let $\Delta_2$ be its anticanonical polytope. 
Let $\Delta_1$ be the lattice polytope whose vertices are $(1,0),(0,1),(-1,0),(0,-1)$. 
The pair $(\Delta_1,\Delta_2)$ is a good pair and the toric variety $Y:=X_{\Delta_1^*}$ is $\pp^1\times\pp^1$. 
The family $\mathcal F_{\Delta_1,\Delta_2^*}$  in $X$ is given by $\alpha_1x_1^2x_2+\alpha_2x_2^2x_1+\alpha_3 x_1x_3^2+\alpha_4x_2x_3^2+\alpha_5 x_1x_2x_3$ where $\alpha_i\in \cc^*$, 
whereas the dual family $\mathcal F_{\Delta_2^*,\Delta_1}$ is defined by the equation $\beta_1y_1^2y_3y_4+\beta_2y_1y_2y_3^2+\beta_3y_2^2y_4^2+\beta_4y_1y_2y_3y_4,\ \beta_i\in\cc^*$ in homogeneous coordinates of $Y$ (in fact up to rescaling the variables 
one can put all coefficients $\beta_i$ to be equal to one).

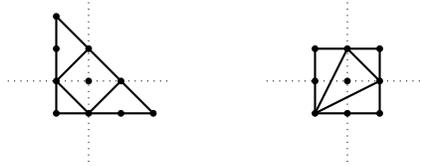
\begin{figure}[h]
\begin{tikzpicture}[scale=.43]
\filldraw [black] 
(0,0) circle (2.5pt) 
(1,0) circle (2.5pt)
(2,-1) circle (2.5pt) 
(1,-1) circle (2.5pt)
(0,-1) circle (2.5pt)
(-1,-1) circle (2.5pt)
(-1,0) circle (2.5pt)
(-1,1) circle (2.5pt)
(-1,2) circle (2.5pt)
(0,1) circle (2.5pt)
; 
\draw[dotted] (-2.5,0)--coordinate (x axis mid) (2.5,0);
\draw[dotted] (0,-2.5)--coordinate (y axis mid) (0,2.5);

\draw [thick] (-1,2)--(2,-1);
\draw [thick] (-1,2)--(-1,-1);
\draw [thick] (-1,-1)--(2,-1);

\draw [thick] (-1,0)--(0,1);
\draw [thick] (-1,0)--(0,-1);
\draw [thick] (0,-1)--(1,0);

\filldraw [black] 
(8,0) circle (2.5pt) 
(7,0) circle (2.5pt)
(7,-1) circle (2.5pt) 
(7,1) circle (2.5pt)
(8,-1) circle (2.5pt)
(8,1) circle (2.5pt)
(9,0) circle (2.5pt)
(9,1) circle (2.5pt)
(9,-1) circle (2.5pt)

; 
\draw[dotted] (5.5,0)--coordinate (x axis mid) (10.5,0);
\draw[dotted] (8,-2.5)--coordinate (y axis mid) (8,2.5);

\draw [thick] (7,1)--(9,1);
\draw [thick] (7,1)--(7,-1);
\draw [thick] (9,1)--(9,-1);
\draw [thick] (7,-1)--(9,-1);

\draw [thick] (8,1)--(7,-1);
\draw [thick] (9,0)--(7,-1);
\draw [thick] (8,1)--(9,0);

  \end{tikzpicture} 
\caption{The good pairs $(\Delta_1,\Delta_2)$ and $(\Delta_2^*,\Delta_1^*)$}
\end{figure}

\end{example}

\begin{remark} 
Families of Calabi-Yau varieties associated to 
pairs of nested reflexive polytopes  
already appeared in the literature in relation with 
the phenomenon of extremal transition \cite{Mo, Rossi, Fr2}.
An extremal transition between two families $\mathcal F$ 
and $\mathcal G$ of  Calabi-Yau manifolds occurs when there 
is a degeneration of $\mathcal F$ to a singular family 
$\mathcal F_0$ whose resolution is isomorphic to $\mathcal G$.
An inclusion of reflexive polytopes $\Delta_1\subset \Delta_2$  can produce 
such a transition since the family $\mathcal F_{\Delta_1,\Delta_2^*}$ is a degeneration 
of $\mathcal F_{\Delta_2, \Delta_2^*}$ (obtained putting some coefficients to zero) 
and is birational to the family  $\mathcal F_{\Delta_1,\Delta_1^*}$.
Such extremal transitions have been used to prove connectedness of 
moduli spaces of Calabi-Yau manifolds \cite{BKK, ACJM, CGGK, Kreuzer1997}.
\end{remark} 

\begin{proposition}\label{bir}
Let $(\Delta_1,\Delta_2)$ and $(\Delta_1',\Delta_2)$ 
be two good pairs. 
Then the dual families of $\mathcal F_{\Delta_1,\Delta_2^*}$ 
and $\mathcal F_{\Delta_1',\Delta_2^*}$ are birational.
\end{proposition}

\begin{proof}
This follows from the fact 
that the toric varieties 
$X_{(\Delta_1')^*}$ and $X_{(\Delta_1)^*}$ 
are compactifications of the same torus 
$T_N={\rm Spec}\,\cc[N]$
and that the dual families in $T_N$ 
are both defined 
by linear combinations of the monomials 
corresponding to the points of the polytope 
$\Delta_2^*$.
\end{proof}

In the next section we will show 
that the duality between good pairs also includes 
Berglund-H\"ubsch-Krawitz duality. 
This implies that Proposition \ref{bir} 
can be seen as a generalization of \cite[Theorem 3.1]{Sh}. 
See also \cite{DFK} for a recent generalization 
of the proposition in a more general setting and 
in terms of derived equivalences.

\section{Berglund-H\"ubsch-Krawitz (BHK) duality}\label{bhk}

\subsection{The BHK construction}\label{sec-bhk}
We will recall a mirror construction due to the physicists 
Berglund and H\"ubsch \cite{BH} and later refined by Krawitz in \cite{K}.
Let $\pp(w)=\pp(w_1,\dots,w_n)$ be a normalized 
weighted projective space and let $W$ be 
a homogeneous polynomial of Delsarte type, 
i.e. having the same number of monomials and variables.
Up to rescaling the variables, we can assume that 
\begin{equation}\label{pot}
W(x_1,\dots,x_n)=\sum_{i=1}^n\prod_{j=1}^n x_j^{a_{ij}},
\end{equation}
so that $W$ is uniquely determined by its matrix of exponents $A=(a_{ij})$.
We will denote by $Y_W$ the hypersurface defined by $W$ in $\pp(w)$ and
we will assume that
\begin{enumerate}
\item $A$ is invertible over $\qq$,
\item $Y_W$ is quasismooth,
\item $\deg(W)=\sum_{i=1}^n w_i$  ({\em Calabi-Yau condition}). 
\end{enumerate}
 The assumptions (ii) and (iii) imply that $Y_W$ 
 is a Calabi-Yau variety by Proposition \ref{prop-canonicalsing} and \cite[Lemma 1.11]{CG}.
 \begin{remark}
By the proof of \cite[Lemma 2]{Sk} the condition of quasismoothness 
implies that the matrix $A$ is invertible over $\qq$. Thus condition 
(i) in the above construction 
is redundant. 
Moreover, it can be easily proved that 
asking $A$ to be invertible over $\qq$ is equivalent 
to say that the convex hull of the elements $u_1,\dots,u_n\in M$ 
corresponding to the monomials of $W$ is a simplex in $M_\rr$.
\end{remark}
 
 If we now consider the transposed matrix of $A$, 
 this defines in the same way a Delsarte type polynomial $W^*$.
 A set of weights $w^*=(w_1^*\dots, w_n^*)$ 
 which makes $W^*$ homogeneous is given by 
 the smallest integer multiple of the vector
 \[
 q^*=(A^T)^{-1}\cdot {\bf 1},
 \]
 where $\bf 1$ denotes the column vector with all entries equal to $1$. 
 By the quasismoothness assumption it follows that 
$w^*$ can be chosen with all positive entries (see Remark \ref{pos}).
 Thus $W^*$ defines a hypersurface $Y_{W^*}$ in $\pp(w^*)$.
 By \cite[Theorem 1]{kreuzerskarke} $W^*$ is still quasismooth and an easy computation shows 
 that it satisfies the Calabi-Yau condition in $\pp(w^*)$.
 Thus $Y_{W^*}$ is a Calabi-Yau variety.
 The Berglund-H\"ubsch-Krawitz construction gives a duality
\[
Y_W/\tilde G \longleftrightarrow Y_{W^*}/\tilde G^*,
\]
where $\tilde G$ denotes a quotient group $G/J$, 
with $J\subseteq {\rm SL}_n(\cc)$ the subgroup of diagonal automorphisms 
inducing the identity on $\pp(w)$ 
and $G$ a subgroup of diagonal automorphisms in ${\rm SL}_n(\cc)$ 
containing $J$ and  acting trivially on $W$, i.e.
\[
W(g(x))=W(x),\ \forall g\in G.
\]
The transposed group $\tilde G^*$ is defined as $G^*/J^*$,
where $J^*$ is the analogous of $J$ for $\pp(w^*)$ and $G^*$ is defined by
\begin{equation}\label{Gstar1}
G^*:=\left\{\prod_{j=1}^{n}(\rho_j^*)^{\alpha_j}\,|\,\prod_{j=1}^n x_j^{\alpha_j} \,\mbox{is }\, G\mbox{-invariant}\right\},
\end{equation}
where $\rho_j^*:={\rm diag}(\exp(2\pi ia^{j1}),\dots,\exp(2\pi ia^{jn}))$ and $a^{ji}$ are the entries of $A^{-1}$.
Several equivalent definitions for the transposed group can be found in \cite[\S 3]{ABS}.
The groups $\tilde G$ and $\tilde G^*$ both act simplectically \cite[Proposition 2.3]{ABS},
thus $Y_W/\tilde G$ and $Y_{W^*}/\tilde G^*$ are both Calabi-Yau varieties. 
In fact we will prove in Proposition \ref{cybh} that they are anticanonical hypersurfaces 
of the $\qq$-Fano toric varieties $\pp(w)/\tilde G$ and $\pp(w^*)/\tilde G^*$.

In \cite[Theorem 2]{CR} Chiodo and Ruan proved that $Y_W/\tilde G$ and $Y_{W^*}/\tilde G^*$ 
have symmetric Hodge diamonds for the Chen-Ruan orbifold cohomology. 

\begin{remark}\label{pos} 
Observe that, by the above definition, we have that
 \[
 A^Tq^*={\bf 1}\Longleftrightarrow \sum_{i=1}^nq_i^*P^T(u_i)=0 \Longleftrightarrow \sum_{i=1}^nq_i^*u_i=0,
 \]
 where $u_1,\dots, u_n\in M$ are 
 the points corresponding to the monomials of $W$,
 i.e. $P^T(u_i)+{\bf 1}$ is the $i$-th row of $A$.
 Moreover
 \[
 \sum_{i=1}^{n} q_i^*={\bf 1}^T(A^T)^{-1}{\bf 1}={\bf 1}\,A^{-1}{\bf 1}=1.
 \]
 Thus the entries of the vector $q^*$ are the barycentric 
 coordinates of the origin in the simplex with vertices $u_1,\dots,u_n$.
 In particular all the entries of $q^*$ are positive if and only if the origin 
 lies in the interior of the simplex.
Since $X_W$ is quasismooth, by \cite[Lemma 2]{Sk} the simplex contains the origin in its interior.
\end{remark}

\subsection{The generalized BHK construction}\label{sec-bhktoric}

We now describe a natural generalization of the BHK construction 
where weighted projective spaces are replaced by their higher class group rank analogs: 
$\qq$-Fano toric varieties with torsion-free class group.

\begin{construction}\label{construction}
Let $X=X_{\Sigma}$ be a $\qq$-Fano toric variety  with torsion free class group.
Denote by $n_1, \ldots, n_r$ the minimal generators of the rays of 
$\Sigma$ and by $\Delta_2$ the anticanonical polytope of $X$. 
Let $\Delta_1 \subseteq \Delta_2$ be a lattice polytope containing 
the origin in its interior and let us denote by $u_1,\ldots, u_s$ its vertices.

Note that by Corollary \ref{canon2} $(\Delta_1, \Delta_2)$ is a good pair and hence  $u_1,\ldots, u_s$ are primitive lattice vectors.
Moreover $n_1, \ldots, n_r$ generate the whole lattice $N$ since $\Cl(X)$ is torsion-free.
Instead of considering a single polynomial $W$ 
(made unique by a rescaling of the variables) as in the original construction, 
we consider a family of polynomials of the form 
\[
W(x_1,\ldots, x_r) = \sum_{i=1}^{s} \alpha_i \prod_{j=1}^{r} x_j^{a_{ij}}, \quad \alpha_i \in \cc^*,
\]
uniquely determined by the matrix 
\[
A=(a_{ij})_{\substack{i=1,\ldots,s \\ j=1,\ldots,r}} \quad \text{ with } a_{ij}=\langle u_i,n_j \rangle +1.
\] 

By construction $W$ is homogeneous of degree $[-K_X]$ and the equation $W=0$ defines a hypersurface 
$Y_W$ of $X$. 
Observe that, by the proof of \cite[Lemma 2]{Sk},  the quasismoothness 
condition for $W$ implies that its Newton polytope contains the origin in its interior.

An important fact to notice is that by the hypotheses 
made on $u_1,\ldots, u_s$, the matrix $A$ determines the variety $X$.

\begin{lemma}\label{recoverX}
Since $\Delta_1$ is a lattice polytope containing the origin in its interior, it is possible to recover the toric variety $X$ from the matrix $A$.
\end{lemma}

\begin{proof}
Since the polytopes $\Delta_1$ and $\Delta_2^*$ are full dimensional, 
we know that the matrix $A-{\bf 1} = (\langle u_i,n_j \rangle)_{i=1,\ldots,s, \ j=1,\ldots,r}$ has rank $n$, so that we may choose $n$ 
linearly independent rows $l_{i_1}, \ldots, l_{i_n}$ of it. 
These define a submatrix $P'$, equal to the product $UP$ 
where the rows of $U$ are the coefficients of $u_{i_1}, \ldots, u_{i_n}$ 
and $P$ is the matrix of the $P$-morphism of $X$ (see \S \ref{background}).   

The lattice $N'=U N$ is a sublattice of $N$ of index $|N/N'|=|\det(U)|$ 
and by \cite[Prop. 3.3.7]{CLS}, the group $K'$ in the short exact sequence  
\begin{equation}\label{seq'}
0 \longrightarrow (N')^\vee=M' \overset{(P')^T}{\longrightarrow} \zz^r \overset{Q'}{\longrightarrow} K' \longrightarrow 0,
\end{equation}
is isomorphic to the class group of the geometric quotient $X/H$ where $H=N/N'$.
Since $P'=UP$ we have $K' \simeq K \times H$ and we can factorize the sequence as 
\begin{equation}\label{longseq}
0 \longrightarrow  M \overset{P^T}{\longrightarrow} M' \overset{U^T}{\longrightarrow} \zz^r \overset{Q'}{\longrightarrow} K' \overset{F}{\longrightarrow} K \longrightarrow 0,
\end{equation}
where $M$ is a sublattice of $M'$ of index $|\det(U)|$ and the map $F : K'\simeq K \times H \to K$ is the first projection.
\smallskip

It follows that knowing just $A$, we can form a short exact 
sequence \eqref{seq'} by finding a suitable submatrix of $A-{\bf 1}$ of rank $n$. 
Then by computing the torsion part of $K'$ we can deduce the map $F$ of \eqref{longseq} 
and hence the sublattice $M=\ker(F \circ Q')$ of $M'$.
Up to a unimodular basis change, this gives $P^T$ 
and hence the minimal generators $n_1, \ldots, n_r$ of $\Sigma$.
This is enough to recover $X$ since it is $\qq$-Fano. 
Indeed by \cite[Prop. 4.3]{Reid83} the cones of the fan $\Sigma$ 
are the cones over the faces of $\Delta_2^*=\Conv(n_1, \ldots, n_r)$. 
\end{proof}
\medskip

Let us now turn to the groups of diagonal automorphisms of $X$ that stabilize $W$.
The stabilizer of the monomials of $W$ under the induced action of the big torus 
$\hat{T}\cong (\cc^*)^r$ on the ring $\cc[x_1, \ldots,x_r]$ is 
\[
\Aut(W)=
\left\{ (\lambda_1, \ldots, \lambda_r) \in \hat{T} \mid
 \ \prod_{j=1}^{r} \lambda_j^{a_{ij}} =1  ,\ \text{for } i=1,\ldots,s\right\}.
\]
Let $\Psi$ and $G_\Sigma=\ker(\Psi)$ as defined in section 1.
The group $\Aut(W)$ does not contain $G_\Sigma$ in general, 
but if we consider the special linear parts  $J_\Sigma=G_\Sigma \cap \Sl_r(\cc)$ 
and $\Sl(W)=\Aut(W) \cap \Sl_r(\cc)$, we have  $J_\Sigma \subset \Sl(W)$.
Indeed, for all $(\lambda_1, \ldots, \lambda_r) \in \hat{T} \cap \Sl_r(\cc)$ 
we have $\prod_{j=1}^{r} \lambda_j =1$ 
and since $a_{ij}=\langle u_i,n_j \rangle +1$ we have
\[
\Sl(W)=\{\lambda\in \hat T\cap \Sl_r(\cc): \Psi(\lambda)(u)=1,\ \forall u\in M_W\},
\]
where $M_W$ is the sublattice of $M$ generated by $u_1, \ldots, u_s$. 
Moreover
 \[
 J_{\Sigma}=\{\lambda\in \hat T\cap \Sl_r(\cc): \Psi(\lambda)(u)=1,\ \forall u\in M\}.
 \]
Note that contrary to the original BHK construction, the group $\Sl(W)$ here can be infinite. 
But when this is the case, then $J_\Sigma$ is infinite as well and the quotient is finite.
 
\begin{lemma}\label{iso}
The homomorphism $\Psi$ defines a perfect pairing
\[
 \bar\Psi:\Sl(W)/J_{\Sigma}\times M/M_W\to \cc^*,\ ([\lambda],[u])\mapsto \Psi(\lambda)(u). 
\]
In particular $\Sl(W)/ J_\Sigma$ is isomorphic to $M/M_W$.
\end{lemma}

\begin{proof}
The homomorphism $\bar\Psi$ is clearly 
well-defined and it induces an injective 
homomorphism from $\Sl(W)/ J_\Sigma$ to 
${\rm Hom}(M/M_W,\cc^*)$ by definition of $J_{\Sigma}$.
Let $(u_1, \ldots, u_n)$ be a basis of $M$ such that there exist 
positive integers $k_1, \ldots, k_n \in \nn$ 
with $(k_1 u_1, \ldots, k_n u_n)$ a basis of $M_W$.
We have $M/M_W \simeq \bigoplus_{i=1}^{n} \zz/k_i\zz$. 
To prove the surjectivity we will prove that 
there exists an element of $\Sl(W)$ whose image is 
the homomorphism sending $u_1$ to $e^{2\pi i/{k_1}}$ 
and $u_i$ to $1$ for $i>1$. 
Consider the dual basis of $(k_1 u_1, \ldots, k_n u_n)$. 
This is a $\zz$-basis of $M_W^\vee$, let us denote it by $(m_1, \ldots, m_n)$. 

Now pick a $r$-uple of rational numbers $w_1, \ldots, w_r \in \qq$ 
such that $\sum_{j=1}^{r} w_j n_j = m_1$ and write $e^{i 2\pi w}=(e^{i 2\pi w_1}, \ldots, e^{2\pi i w_r})$. 
Since $m_1$ is an element of $M_W^\vee$, we have 
$$
\sum_{j=1}^{r} w_j \langle n_j, u \rangle = \langle m_1, u \rangle \in \zz \quad \text{ for all } u \in M_W,
$$ 
which implies that $e^{2\pi i w} \in \Sl(W)$. 
It remains to observe that 
$$
e^{2\pi i \sum_{j=1}^{r} w_j \langle u_1,n_j \rangle}=e^{2\pi i/k_1},\quad e^{2i \pi \sum_{j=1}^{r} w_j \langle u_i,n_j \rangle}=1 \text{ for } i>1.
$$
A similar argument holds for any $u_i$, proving the surjectivity.
\end{proof}

We now choose a subgroup $G$ of $\Sl(W)$ containing $J_\Sigma$, 
or equivalently a subgroup $\tilde{G}=G/ J_\Sigma$ of $\Sl(W)/ J_\Sigma$.
Observe that by Lemma \ref{iso} there exists a unique lattice $M_G$ 
satisfying $M_W \subseteq M_G \subseteq M$ such that 
\begin{equation}\label{G2}
G=\left\{ \lambda \in \hat{T} \cap \Sl_r(\cc) \mid \Psi(\lambda)(u)=1 ,\ \forall \ u \in M_G \right\}.
\end{equation}
We thus define an orbifold $Y_{W,G}=Y_W/\tilde{G}$, 
the family of which we denote by $\F(A,\tilde G)$. 

The following generalizes \cite[Lemma 2.3]{Sh}.

\begin{proposition}\label{prim}
Let $X,G$ be as in Construction \ref{construction}.
Then the quotient map $X\to X/G$ 
induces the identity $\mathcal R(X/G)\to \mathcal R(X)$
and $\Cl(X/G)\cong \Cl(X)\oplus G$.
\end{proposition}

\begin{proof}
By Lemma \ref{quot} it is enough 
to prove that the quotient map 
is induced by a lattice monomorphism 
$N\to N_G$ such that 
the primitive generators $n_1,\dots,n_r\in N$
of the rays of the fan of $X$ remain primitive in $N_G$.
In fact, assume that $n_i=kn$ for some $n\in N_G$ 
and some integer $k$. 
Since $N_G\subset M_W^\vee$, then 
$a_{ij}-1=\langle n_i,u_j\rangle=k\langle n,u_j\rangle\in k\zz$ 
for all $j$.
Since a monomial in $W$ can not contain all the variables 
(otherwise it would correspond to $0\in M$),
this gives $k=1$.
\end{proof}


\begin{proposition}\label{cybh}
The orbifolds $Y_{W,G}$ are Calabi-Yau varieties for general $\alpha_i\in \cc^*$.
\end{proposition}
\begin{proof}
By construction $Y_W$ is an anticanonical hypersurface in $X$. 
The quotient $Y_W\to Y_W/G$ is clearly induced by the quotient 
$X\to X/G$ induced by the inclusion $N\subseteq N_G=M_G^{\vee}$. 
Let $\Delta_2\subset M_\qq$ be the anticanonical polytope of $X$ 
and $\Delta_1$ be the Newton polytope of $Y_W$. 
The anticanonical polytope of $X/G$ is equal to
$\Delta_2\subset (M_G)_\qq=M_\qq$ by 
Proposition \ref{prim}.
The monomials in the defining 
equation of $Y_W/G$ are the lattice points 
$u_i\in M_W\subset M_G$.
Thus, the Newton polytope of 
$Y_W/G$ is clearly $\Delta_1$.
Since $(\Delta_1,\Delta_2)$ is a good pair in $M_{\qq}$ by Corollary \ref{canon2}, 
then the same holds 
in $(M_G)_{\qq}$ by Lemma \ref{gp}.
By Theorem \ref{cy} and  Remark \ref{genhyp} the orbifold $Y_{W,G}$,
whose monomials in Cox coordinates correspond to the 
vertices of $\Delta_1$, are Calabi-Yau varieties.
\end{proof}

Let us now define the dual family $\F(A^T,\tilde G^*)$.
Let $X^*$ be the toric variety  whose fan is the collection 
of cones over the faces of $\Delta_1$, 
considered with respect to the lattice $M_W$ generated by $u_1, \ldots, u_s$.
 The transposed matrix $A^T$  of $A$ gives the family  
$$
W^*(y_1,\ldots, y_s) = \sum_{i=1}^{r} \beta_i \prod_{j=1}^{s} y_j^{a_{ji}}, \quad \beta_i \in \cc^*,
$$
where $y_1,\dots,y_s$ are the Cox coordinates of $X^*$.
Observe that by  Lemma \ref{recoverX} the toric variety $X^*$ can 
be recovered uniquely by the matrix $A^T$.

We now define the transposed of the group $G$.
If $M^\vee \subseteq M_G^\vee \subseteq M_W^\vee$ 
denote the dual lattices and $\hat{T}^\vee=(\cc^*)^s$ is an algebraic torus of dimension $s$, 
we define  
\begin{equation}\label{Gstar2}
G^*:=\left\{ \nu \in \hat{T}^\vee \cap \Sl_s(\cc) \mid \Psi^*(\nu)(n)=1,\ \text{for all} \ n \in M_G^\vee \right\},
\end{equation}
where 
\[
\Psi^*:\hat{T}^\vee\to {\rm Hom}(N,\cc),\quad  \nu\mapsto (n\mapsto \prod_{i=1}^s\nu_i^{\langle u_i,n\rangle}).
\]
  
Let us show that this generalizes the original definition of the transposed group given by Krawitz  \eqref{Gstar1} 
(see also \cite[Proposition 2.3.1]{Bor}).
 
\begin{lemma}\label{GTG*}
When the good pair $(\Delta_1, \Delta_2)$ is formed by simplices then the definitions \eqref{Gstar1} and \eqref{Gstar2} coincide.
\end{lemma}

\begin{proof}
The proof relies on the fact that when $r=s=n+1$ the diagonal groups 
preserving the polynomials $W$ and $W^*$ are finite, hence formed of $(n+1)$-uples of roots of unity. 
Let us write $e^{2\pi i v}=(e^{2\pi i v_0}, \ldots, e^{2\pi i v_{n}})$ for all $v \in \qq^{n+1}$.  
Given a subgroup $G$ with $J_{\Sigma}\subset G\subset \Sl(W)$ we will denote by 
$G_1^*$ its transposed as defined in \eqref{Gstar1} and $G_2^*$ its transposed 
as in \eqref{Gstar2}. 
It is proven in \cite[\S 3.5]{ABS} that  
$$
G^*_1:=\lbrace e^{2\pi i v} \in \hat{T}^\vee \cap \Sl_{n+1}(\cc) \mid v^T A w \in \zz \ \text{for all} \ e^{2\pi i w} \in G \rbrace.
$$
We first remark that for all $(v,w) \in \qq^{n+1} \times \qq^{n+1}$ we have
$$
e^{2\pi i v}, e^{2\pi i w} \in \Sl_{n+1}(\cc) \ \Rightarrow \ \sum_{i,j=0}^n v_i w_j  \in \zz,
$$
and that by \eqref{G2}, for all $w \in \qq^{n+1}$ such that $e^{2\pi i w} \in \Sl_{n+1}(\cc)$ we have 
$$
e^{2\pi i w} \in G 
\ \Leftrightarrow \ 
\forall \ u \in M_G,\ \sum_{j=0}^n \langle u,n_j \rangle w_j \in \zz  
\ \Leftrightarrow \ 
\sum_{j=0}^n w_j n_j \in M_G^\vee.   
$$
It follows that for all $v \in \qq^{n+1}$ such that $e^{2\pi i v} \in \Sl_{n+1}(\cc)$ we have
\begin{eqnarray*}
e^{2\pi i v} \in G^*_1 
& \Leftrightarrow &
\sum_{i,j=0}^n  v_i (\langle u_i,n_j \rangle +1) w_j \in \zz \ \text{for all} \ e^{2\pi i w} \in G   \\
  & \Leftrightarrow &
\langle \sum_{i=0}^n v_i u_i, \sum_{j=0}^n w_j n_j \rangle \in \zz \ \text{for all} \ e^{2\pi i w} \in G    \\
  & \Leftrightarrow &
\sum_{i=0}^n v_i u_i \in M_G.
\end{eqnarray*}
On the other hand, writing $\nu= e^{2\pi iv}$, definition \eqref{Gstar2} is equivalent to
$$
G^*_2:=\lbrace e^{2\pi i v} \in \hat{T}^\vee \cap \Sl_{n+1}(\cc) \mid 
\sum_{i=0}^n  v_i \langle u_i,n \rangle \in \zz \ \text{for all} \ n \in M_G^\vee \rbrace.
$$
Since for all $v \in \qq^{n+1}$ we have 
$$
\sum_{i=0}^n  v_i \langle u_i,n \rangle \in \zz \ \text{for all} \ n \in M_G^\vee 
\ \Leftrightarrow \ 
\sum_{i=0}^n v_i u_i \in M_G,   
$$
we get $G^*_1=G^*_2$ so that the two definitions are indeed equivalent.
\end{proof} 

\begin{remark}
In \cite{ABS} it is proved that $G^*$ is the orthogonal complement of $G$ 
with respect to the bilinear form
\[
\Sl(W)/J_W\times \Sl(W^*)/J_{W^*}\to \cc^*,\ (e^{2\pi i v},e^{2\pi i w})\mapsto e^{2\pi iv^TAw}.
\]
It can be easily proved that the same description can be given for the generalized 
Berglund-H\"ubsch-Krawitz construction.
\end{remark}

\begin{definition}
The {\em generalized Berglund-H\"ubsch-Krawitz dual family} of the family $\F(A,\tilde G)$
is the family $\F(A^T,\tilde {G}^*)$ formed by the orbifolds 
\[
Y_{W^*,\tilde{G}^*}=\lbrace W^*=0 \rbrace /\tilde{G}^*.
\]
\end{definition}
\end{construction}

We want to emphasize that the above construction is a direct generalization of the original BHK construction, 
simply replacing canonical simplices (with the quasismoothness condition) by arbitrary canonical polytopes.
It is quite striking that the transposition rule extends naturally to the case of a non square matrix, 
thanks to Cox construction of homogeneous coordinates on the ambient toric varieties.

We conclude with the proof of Theorem 2, which shows that the generalized BHK construction can be described 
as a duality of good pairs.

\begin{proof}[Proof of Theorem 2] 
Let $\Delta_2\subset M_{\qq}$ be the anticanonical polytope of $X$ 
and $\Delta_1\subset \Delta_2$ be the Newton polytope of $Y_W$.
We recall that  the anticanonical polytope of $X/\tilde G$ 
is the same polytope $\Delta_2\subset (M_G)_{\qq}$
and the Newton polytope of $Y_{W,G}$ is 
$\Delta_1\subset (M_G)_{\qq}$.
Thus $(\Delta_1,\Delta_2)$ is a good pair in $(M_G)_\qq$ as already explained
in the proof of Proposition \ref{cybh}.
We now consider the polar pair $(\Delta_2^*,\Delta_1^*)$ as a good pair 
in $(M_W^{\vee})_{\qq}$.
By construction $X^*$ is defined by the fan over the faces of $\Delta_1$ 
and the vertices of $\Delta_2^*$  correspond to the monomials of $W^*$.
Thus $\Delta_1^*$ is the anticanonical polytope of $X^*$ and $\Delta_2^*$
is the Newton polytope of $Y_{W^*}$.   
Looking at the definition of the transposed group $G^*$ one immediately sees 
that the pair $(\Delta_2^*,\Delta_1^*)$ seen in $(M_G^{\vee})_{\qq}$
 defines the toric variety $X^*/\tilde G^*$ and the hypersurface $Y_{W^*,G^*}$.
 \end{proof}

 \begin{example}\label{bv}
Let $X=\pp(5,5,4,4,2)$ and let $\Delta_1\subset M_\qq$ be the convex hull of the 
lattice points $u_1,\ldots,u_5\in M$ corresponding to the monomials
$$x_1^3x_2,\ x_2^4,\ x_3^5,\ x_4^5,\ x_5^{10}.$$ 
The polytope $\Delta_1$ is a simplex. We denote by $\Delta_2\subset M_\qq$ 
the anticanonical polytope of $X$, thus $(\Delta_1,\Delta_2)$ is a good pair and they both are reflexive.
Since $\Delta_1$ and $\Delta_2$ are simplices, then the diagonal group of automorphisms preserving 
\[
W=x_1^3x_2+x_2^4+x_3^5+x_4^5+x_5^{10}
\] 
is finite and $\Sl(W)/J_W\cong \zz/5\zz$.  
The polynomial $W$ is quasismooth since it satisfies the conditions of \cite[Theorem 1]{kreuzerskarke}. 
The  BHK construction  gives a dual  Calabi-Yau hypersurface in the toric variety $Y$ 
which is the quotient of $\pp(10,6,6,5,3)$ by the group of order five generated by $g=(1,e^{2\pi i\frac{ 1}{5}},e^{2\pi i\frac{4}{5}},1,1)$.
The hypersurface is  defined by 
the following  polynomial in the Cox coordinates of $Y$:
$$
W^*=y_1^3+y_1y_4^4+y_2^5+y_3^5+y_5^{10}.
$$ 
The threefold defined by $W$ in $X$ is a birational projective model for the Borcea-Voisin Calabi-Yau threefold obtained from the elliptic curve $E:=\{x_0^2+x_1^3x_2+x_2^4=0\}\subset\pp(2,1,1)$ 
and the K3 surface $S:=\{y_0^2+y_1^5+y_2^5+y_3^{10}=0\}\subset\pp(5,2,2,1)$ (see \cite[Proposition 4.4]{ABS2}).
\end{example}

\begin{example}\label{bv2}
Let $X=\pp(5,5,4,4,2)$ as in Example \ref{bv} and let
 $\Delta'_1\subset \Delta_2$ be the convex hull of the lattice points $u_1,\ldots,u_8\in M$ corresponding to the monomials
$$x_1^3x_2,\ x_2^4,\ x_3^5,\ x_4^5,\ x_5^{10},\ x_1^2x_5^5,\ x_1^2x_4^2x_5,\ x_1^2x_3^2x_5.$$
Observe that  $\Delta_1\subset\Delta'_1$ 
and $\Delta'_1$ is not a simplex. 
We have that $(\Delta_1',\Delta_2)$ is a good pair of reflexive polytopes and one can prove that the general 
element of $\mathcal F_{\Delta'_1,\Delta_2^*}$ is quasismooth. 
The lattice   generated by $u_1, \ldots, u_8$ is equal to $M$ so that the only 
possible choice for $\tilde G$ is the trivial group.
Thus starting from $\Delta'_1$ one obtains the family of Calabi-Yau hypersurfaces $\mathcal F(A',G)$
\[
\alpha_1x_1^3x_2+\alpha_2x_2^4+\alpha_3x_3^5+\alpha_4x_4^5+\alpha_5x_5^{10}+\alpha_6x_1^2x_5^5+\alpha_7x_1^2x_4^2x_5+\alpha_8 x_1^2x_3^2x_5=0.
\]
The generalized BHK construction gives the dual family $\mathcal F((A')^T, G^*)$ 
\[
 \beta_1y_1^3y_3^2y_4^2y_5^2+\beta_2 y_1y_2^4+\beta_3 y_3^5y_4y_5y_6^{10}+\beta_4 y_4^2y_7^5+\beta_5y_5^2y_8^5=0
\]
in the toric variety 
whose Cox ring is $\cc[y_1,\dots,y_8]$ with Class group isomorphic to $\zz^4$ and grading matrix
\[
\left(
\begin{matrix}
    1 &  1&  1&  0&  0&  0&  1&  1\\
   0&  5&  0&  0& 10&  1&  4&  0\\
    0&  5&  0&  5&  5&  1&  2&  2\\
    0&  7&  1&  4&  9&  1&  4&  2\\
    \end{matrix}
    \right).
\]
\end{example}

\section{About the stringy Hodge numbers of good pairs}    
The polar duality between good pairs provides 
a duality between families of Calabi-Yau hypersurfaces of $\qq$-Fano toric varieties. 
Of course a natural question arises:
\vspace{0.1cm}
 
\noindent {\em Question.}
Do the families $\mathcal F_{\Delta_1,\Delta_2^*}$ and 
$\mathcal F_{\Delta_2^*,\Delta_1}$ associated to a good pair $(\Delta_1,\Delta_2)$
satisfy the topological mirror test, i.e. 
\[
h^{p,q}_{st}(X)=h^{n-p,q}_{st}(X^*),\quad 0\leq p,q\leq n\ 
\]
for general $X\in \mathcal F_{\Delta_1,\Delta_2^*}$ and $X^*\in \mathcal F_{\Delta_2^*,\Delta_1}$?
 \vspace{0.1cm}
 
This is known to be true for anticanonical hypersurfaces of toric Fano varieties 
(i.e. when $\Delta_1=\Delta_2)$ by Batyrev and Borisov \cite{BaBo} 
and in the classical  Berglund-H\"ubsch case by Chiodo and Ruan \cite{CR} 
(i.e. when $\Delta_1,\Delta_2$ are simplices and quasismoothness holds)
in terms of Chen-Ruan orbifold cohomology.
Observe that  in the latter case orbifold Hodge numbers are known to be the same as the stringy
Hodge numbers (see \cite{Yasuda} and \cite{BM}). 
Unfortunately we are unable to give a complete answer at the moment.
Example \ref{fred}  shows that the answer is negative in general 
and Example \ref{bvhodge} leads us to think that the answer could be 
positive if quasismoothness holds.

We recall that in case the polytope is reflexive of dimension four Hodge numbers 
can be computed in terms of combinatorial properties of the polytope 
\cite[Corollary 4.5.1]{Ba}.

The following result allows to compute the stringy Hodge 
numbers for good pairs of four dimensional reflexive polytopes.
We recall that, given a reflexive polytope $\Delta$, 
an {\em MPCP resolution} of the toric variety 
$X_{\Delta}$ is defined by a complete, simplicial projective fan  
which is a subdivision of the normal fan $\Sigma_{\Delta}$ 
to $\Delta$ and whose rays are generated by nonzero lattice points of $\Delta^*$. 
In \cite{Fr} the author introduces the more general notion of 
{\em $\Delta$-maximal fan}, which has the same definition of 
an MPCP resolution except for the fact that it is not necessarily projective 
and it does not need to refine $\Sigma_{\Delta}$. 
By \cite[Corollary 4.2.3]{Ba} an MPCP resolution induces a crepant resolution 
of the general anticanonical hypersurface of $X_{\Delta}$ in case $\Delta$ is four dimensional.
 By \cite[Theorem 4.9]{Fr}, given any $\Delta$-maximal fan $\Sigma$ for a four dimensional 
reflexive polytope $\Delta$, the general member of the anticanonical linear series 
of the toric variety $X_{\Sigma}$ is smooth.

\begin{proposition}\label{hn} Let $\Delta_1\subseteq \Delta_2$ be four dimensional 
reflexive polytopes. The general elements of the families $\mathcal F_{\Delta_1,\Delta_2^*}$  
and $\mathcal F_{\Delta_1,\Delta_1^*}$ have the same Hodge numbers.
\end{proposition}

\begin{proof} 
Let $\Sigma_2$ be the normal fan to $\Delta_2$ and 
$\Sigma_2'$ be the fan of an MPCP resolution of $X_{\Delta_2}$.   
By \cite[Lemma 6.1]{Fr}  there exists a projective
 $\Delta_1$-maximal fan $\tilde \Sigma_2$ which refines $\Sigma_2'$, and thus $\Sigma_2$.
Let  $\phi : X_{\tilde \Sigma_2} \to X_{\Sigma_2}$ be the corresponding toric morphism.
By \cite[Theorem 4.9]{Fr} the general member $\tilde Y$ of the anticanonical linear series 
of $X_{\tilde \Sigma_2}$ is smooth. Moreover, by Corollary \ref{crep}, $\tilde Y$ is a crepant resolution of 
the general member $Y$ of $\mathcal F(\Delta_1)$. 
Now let $\tilde Z$ be a crepant resolution of the general anticanonical 
hypersurface $Z$ of $X_{\Delta_1}$, induced by an MPCP resolution.
By \cite[Theorem 3.12]{Bstringy}  $\tilde Y$ and $Y$ have the same stringy Hodge numbers,
and the same holds for $\tilde Z$ and $Z$.  
Since $\tilde Y$ and $\tilde Z$ are birational smooth Calabi-Yau threefolds 
(since they are given by the same Newton polytope),
then they have the same Hodge numbers \cite[Theorem 1.1]{Bbetti}, which achieves the proof.
\end{proof}
 
\begin{corollary}\label{corhn} Let $(\Delta_1,\Delta_2)$ be a good pair of reflexive, four dimensional 
polytopes. The associated families $\mathcal F_{\Delta_1,\Delta_2^*}$ and $\mathcal F_{\Delta_2^*,\Delta_1}$ 
satisfy the topological mirror test if and only if the general elements of the families 
$\mathcal F_{\Delta_1,\Delta_1^*}$ and $\mathcal F_{\Delta_2,\Delta_2^*}$ have the same Hodge numbers.
\end{corollary} 

\begin{proof} It follows from Proposition \ref{hn} and \cite[Corollary 4.5.1]{Ba}, 
which gives that $h^{1,1}(\Delta_2)=h^{2,1}(\Delta_2^*)$ and 
viceversa. \end{proof}
 
\begin{example}\label{fred}
Consider the following four dimensional reflexive polytopes:
\begin{align*}
\Delta_1={\rm Conv}&((1,0,0,0),(0,1,0,0),(0,0,1,0),(0,0,0,1),(-1,-1,-1,-1)),\\
\Delta_2={\rm Conv}&((1,0,0,0),(0,1,0,0),(0,0,1,0),(0,0,0,1),\\&(-1,-1,-1,-1),(1,1,1,1),(0,0,0,-1)).
\end{align*}
By Proposition \ref{hn} and 
\cite[Corollary 4.5.1]{Ba} the Hodge numbers of $(\Delta_1,\Delta_2)$ are $(h^{1,1},h^{2,1})=(101,1)$ 
(observe that the variety associated to the normal fan to $\Delta_1$ is $\pp^4$), 
whereas the Hodge numbers of the dual pair are $(3,79)$.
Thus the dual families $\mathcal F_{\Delta_1,\Delta_2^*}$ and $\mathcal F_{\Delta_2^*,\Delta_1}$ 
do not satisfy the topological mirror test.
A direct computation shows the general member of $\mathcal F_{\Delta_1,\Delta_2^*}$ is not quasismooth.
This example appears in \cite[\S 3]{Fr2}.
\end{example}

\begin{example}\label{bvhodge}
Let $\mathcal F_{\Delta'_1,\Delta_2^*}$ be the family of hypersurfaces of $\pp(5,5,4,4,2)$ 
associated to the reflexive polytope $\Delta'_1\subset \Delta_2$ 
defined in Example  \ref{bv2}.
The Hodge numbers of $\Delta'_1$ and $\Delta_2$ are equal 
to $(h^{1,1},h^{2,1})=(15,39)$  by Batyrev formulas.
Thus by Corollary \ref{corhn} the dual families $\mathcal F_{\Delta'_1,\Delta_2^*}$ 
and $\mathcal F_{\Delta_2^*,\Delta_1'}$ satisfy the topological mirror test.
\end{example}

Combinatorially,  the duality of good pairs can be seen as included in the unified setting 
for mirror symmetry sketched in \cite[\S 7]{Bor}, except for the regularity 
condition described in \cite[Proposition 7.1.3]{Bor}.
We intend to explore such condition, in terms 
of combinatorial properties of polytopes, 
in a further work.

\vspace{0.3cm}

\noindent {\em Acknowledgments.}
 It is a pleasure to thank  Gilberto Bini, Samuel Boissi\`ere, 
 Antonio Laface and Alessandra Sarti for several useful discussions.

\bibliographystyle{plain}
\bibliography{Bibliobh}

\end{document}